\documentclass[12pt]{article}
\usepackage{fullpage}
\usepackage{mathrsfs,pstricks,ifpdf,tikz}
\usepackage{amsfonts}
\usepackage{enumerate}

\usepackage{makecell}

\usepackage{latexsym}
\usepackage{amsmath,amssymb}
\usepackage{amsthm}
\usepackage{ifthen,calc}
\usepackage{color}
\usepackage{graphicx}
\usepackage{latexsym}
\usepackage{multirow}
\usepackage{diagbox}
\usepackage{float}
\usepackage[format=hang, margin=10pt]{caption}

\newtheorem{theorem}{Theorem}[section]
\newtheorem{proposition}[theorem]{Proposition}
\newtheorem{corollary}[theorem]{Corollary}

\newtheorem{lemma}[theorem]{Lemma}

\theoremstyle{definition}
\newtheorem{definition}[theorem]{Definition}
\newtheorem{example}[theorem]{Example}
\newtheorem{question}[theorem]{Question}

\begin{document}

\title{Bounds for DP color function and canonical labelings}

\author{
Ziqing Li\thanks{School of Mathematics, Tianjin University, Tianjin, China: lhy\underline{ }nxz@163.com.}\,
\qquad
Yan Yang\thanks{School of Mathematics, Tianjin University, Tianjin, China: yanyang@tju.edu.cn.
Supported by National Natural Science Foundation of China under Grant 11971346. Corresponding author.}\,
}

\date{}

\maketitle

\begin{abstract}
The DP-coloring is a generalization of the list coloring, introduced by Dvo\v{r}\'{a}k and Postle. Let $\mathcal{H}=(L,H)$ be a cover of a graph $G$ and $P_{DP}(G,\mathcal{H})$ be the number of $\mathcal{H}$-colorings of $G$. The DP color function $P_{DP}(G,m)$ of $G$, introduced by Kaul and Mudrock, is the minimum value of $P_{DP}(G,\mathcal{H})$ where the minimum is taken over all possible $m$-fold covers $\mathcal{H}$ of $G$.
For the family of $n$-vertex connected graphs, one can deduce that trees maximize the DP color function, from two results of Kaul and Mudrock.
In this paper we obtain tight upper bounds for the DP color function of $n$-vertex $2$-connected graphs. Another concern in this paper is the canonical labeling in a cover. It is well known that if an $m$-fold cover $\mathcal{H}$ of a graph $G$ has a canonical labeling, then $P_{DP}(G,\mathcal{H})=P(G,m)$ in which $P(G,m)$ is the chromatic polynomial of $G$. However the converse statement of this conclusion is not always true. We give examples that for some $m$ and $G$, there exists an $m$-fold cover $\mathcal{H}$ of $G$ such that $P_{DP}(G,\mathcal{H})=P(G,m)$, but $\mathcal{H}$ has no canonical labelings. We also prove that when $G$ is a unicyclic graph or a theta graph,  for each $m\geq 3$, if $P_{DP}(G,\mathcal{H})=P(G,m)$, then $\mathcal{H}$ has a canonical labeling.

\medskip

\noindent {\bf Keywords:} DP-coloring; DP color function; $2$-connected graph; ear decomposition; canonical labeling.

\smallskip
\noindent {\bf Mathematics Subject Classification (2020):} 05C15, 05C30, 05C31.

\end{abstract}

\section{Introduction}

All graphs considered in this paper are finite and simple. The set of natural numbers is $\mathbb{N}=\{1,2,3,\dots\}$. For $m\in \mathbb{N}$, let $[m]=\{1,\ldots,m\}$. For any graph $G$, let $V(G)$ and $E(G)$ be its vertex set and edge set respectively. For any $v\in V(G)$, let $N_G(v)$ denote the set of neighbors of $v$ in $G$ and $d_G(v)$ denote the degree of $v$ in $G$. A {\it proper $m$-coloring} of $G$ is a mapping $c: V(G)\rightarrow [m]$ such that $c(u)\neq c(v)$ whenever $uv\in E(G)$. In 1912, Birkhoff \cite{B1912} introduced a function $P(G,m)$ which counts the number of proper $m$-colorings of $G$, it is a polynomial in $m$ and called {\it chromatic polynomial} of $G$. The book by Dong, Koh and Teo~\cite{DKT05} gives an overview for chromatic polynomial problems.

Since it is difficult to get a simple expression for the chromatic polynomial of an arbitrary graph, the bounds for the chromatic polynomials of graphs are of particular interest. For $n$-vertex connected graphs, the upper bound for their chromatic polynomials can be found in \cite{DKT05}.

\begin{theorem}[\cite{DKT05}, Theorem 15.3.2]\label{cp1} Let $G$ be a connected graph with $n$ vertices. Then for all $m\in \mathbb{N}$,
$$P(G, m)\leq m(m-1)^{n-1},$$ where equality holds for $m\ge 3$ if and only if $G$ is a tree.\end{theorem}

For the family of $2$-connected graphs, Tomescu obtained the following result.

\begin{theorem}[\cite{Tomescu1994}, Theorem 2.1]\label{cp2} Let $G$ be a $2$-connected graph with $n$ vertices, where $n\geq 3$. Then for all $m\in \mathbb{N}$ with $m\geq 3$, $$P(G,m)\leq (m-1)^n+(-1)^n(m-1),$$ where equality holds if and only if $G\cong C_{n}$; or $G\cong K(2,3)$ for the case that $n=5$ and $m=3$.
\end{theorem}

In \cite{Felix2019}, Felix gave a survey on the upper bounds for the chromatic polynomials of graphs of given order and size. Recently, some authors focus on the upper bounds for the chromatic polynomials of $n$-vertex graphs with chromatic number $k$, they obtained some inspired results, see \cite{EEFH21, Erey08, Erey18, FHM19, KM19, KM20} for example.

In this paper we obtain analogous results to that in Theorems \ref{cp1} and \ref{cp2}, in the context of DP-coloring. The DP-coloring (also called corresponding coloring) is a generalization of the list coloring, introduced by Dvo\v{r}\'{a}k and Postle \cite{Dvorak2018}.

\begin{definition}[\cite{Dvorak2018}] Let $G$ be a graph. If $X,Y\subseteq V(G)$, we use $G[X]$ for the subgraph of $G$ induced by $X$, and we use $E_{G}(X,Y)$ for the subset of $E(G)$ with one endpoint in $X$ and one endpoint in $Y$. Given a set $S$, $\mathcal{P}(S)$ is the power set of $S$.

\begin{itemize}

\item A \emph{cover} of a graph $G$ is a pair $\mathcal{H}=(L,H)$ consisting of a graph $H$ and a function $L: V(G)\to \mathcal{P}(V(H))$ satisfying the following four requirements:

\noindent(1) the set $\{L(u):u\in V(G)\}$ is a partition of $V(H)$;

\noindent(2) for every $u\in V(G)$, the graph $H[L(u)]$ is complete;

\noindent(3) if $E_{H}(L(u),L(v))$ is nonempty, then $u=v$ or $uv\in E(G)$;

\noindent(4) if $uv\in E(G)$, then $E_{H}(L(u),L(v))$ is a matching (the matching may be empty).

  \item A cover $\mathcal{H}=(L,H)$ is \emph{$m$-fold} if $\vert L(u)\vert =m$ for each $u\in V(G)$, and $\mathcal{H}$ is \emph{full} if for each $uv\in E(G)$, $E_{H}(L(u),L(v))$ is a perfect matching.
  \item An \emph{$\mathcal{H}$-coloring} of $G$ is an independent set in $H$ of size $\vert V(G)\vert$.

  \item The \emph{DP-chromatic number} of $G$, denoted by $\chi_{DP}(G)$, is the smallest $m\in \mathbb{N}$
such that $G$ admits an $\mathcal{H}$-coloring for every $m$-fold cover $\mathcal{H}$ of $G$.
\end{itemize}

\end{definition}

In 2021, Kaul and Mudrock \cite{Kaul2021} gave the definition of DP color function.

\begin{definition}[\cite{Kaul2021}] Let $\mathcal{H}=(L,H)$ be a cover of a graph $G$. We denote $P_{DP}(G,\mathcal{H})$  the number of $\mathcal{H}$-colorings of $G$. The \emph{DP color function} of $G$, denoted by $P_{DP}(G,m)$, is the minimum value of $P_{DP}(G,\mathcal{H})$
where the minimum is taken over all possible $m$-fold covers $\mathcal{H}$ of $G$.
\end{definition}

By the definition of chromatic polynomial and DP color function, $P_{DP}(G,m)\leq P(G,m)$ holds for any graph $G$ and $m\in \mathbb{N}$. Let $\mathcal{H}=(L,H)$ be an $m$-fold cover of a graph $G$. We say that $\mathcal{H}$ has a {\it canonical labeling} if it is possible to name the vertices of $H$ so that $L(u)=\{(u,j):j\in [m]\}$ for each $u\in V(G)$ and $(u,j)(v,j)\in E(H)$ for each $j\in [m]$  whenever $uv\in E(G)$. It is well known that if $\mathcal{H}$ has a canonical labeling, then $P_{DP}(G,\mathcal{H})=P(G,m)$ holds for each $m\in \mathbb{N}$. So there is a natural question as following.

\begin{question}\label{q} Suppose that $\mathcal{H}=(L,H)$ is a full $m$-fold cover of a graph $G$ and $P_{DP}(G, \mathcal{H})=P(G,m)$. Does $\mathcal{H}$ have a canonical labeling?
\end{question}

By finding two examples, we give a negative answer to this question. But, considering our examples are only for some small $m$, we have the following question.

\begin{question}\label{p}
Suppose that $\mathcal{H}=(L,H)$ is a full $m$-fold cover of a graph $G$, does there exist some $M\in \mathbb{N}$ such that for each $m\geq M$, if $P_{DP}(G,\mathcal{H})=P(G,m)$, then $\mathcal{H}$ has a canonical labeling?
\end{question}

Question \ref{p} is closely related to (but not equivalent to) the Problem $3$ in \cite{DY21}. In order to compare DP color functions with chromatic polynomials, Dong and Yang \cite{DY21} defined a class of graphs called $\mathcal{DP^*}$. It denotes the set of graphs $G$ for which there exists $M\in \mathbb{N}$ such that for every $m$-fold cover $\mathcal{H}=(L,H)$ of $G$, if $\mathcal{H}$ has no canonical labelings, then $P_{DP}(G,\mathcal{H})>P(G,m)$ holds for all $m\geq M$. Then, for each graph $G\in\mathcal{DP^*}$, $P_{DP}(G,m)=P(G,m)$ holds when $m\geq M$. And they posed the following question, i.e., the Problem $3$ in \cite{DY21}.

\begin{question}[\cite{DY21}]\label{p1} For a graph $G$,  is  it  true that if there exists an $M\in \mathbb{N}$ such that $P_{DP}(G,m)=P(G,m)$ holds for all $m\geq M$, then $G\in\mathcal{DP^*}$ ?

\end{question}

In \cite{DY21}, some types of graphs have been proved belonging to $\mathcal{DP^*}$.
For example, if a graph $G$ contains a spanning tree $T$ such that for each $e\in E(G)\setminus E(T)$, $\ell(e)$ is odd and $e$ is contained in a cycle $C$ of length $\ell(e)$ with the property that $\ell(e')<\ell(e)$ holds for each $e'\in E(C)\setminus (E(T)\cup \{e\})$, then $G\in\mathcal{DP^*}$, where $\ell(e)=\infty$ if $e$ is a bridge of $G$, and $\ell(e)$ is the length of a shortest cycle in $G$ containing $e$ otherwise. By the definition of $\mathcal{DP^*}$, we have the following result.
\begin{proposition}\label{81}
If a graph $G\in \mathcal{DP^*}$, then there exists an $M\in \mathbb{N}$
such that for any full $m$-fold cover $\mathcal{H}$ of $G$ with $P_{DP}(G,\mathcal{H})=P(G,m)$, where $m\geq M$, $\mathcal{H}$ has a canonical labeling.
\end{proposition}

In Section $2$, we obtain tight upper bounds for the DP color function of $n$-vertex $2$-connected graphs, and give two new proofs for the upper bounds for the DP color function of $n$-vertex connected graphs. In Section $3$, we give two examples that $P_{DP}(G,\mathcal{H})=P(G,m)$ but $\mathcal{H}$ has no canonical labelings, and also give positive answers to Question \ref{p} for unicyclic graphs and theta graphs.

We also note that Kaul, Mudrock, and their coauthors obtain lots of results on DP color function, see \cite{BHK22, HKL20, KMSS21, KMMT21, MT21, Mu22} for example, they study the asymptotics of $P(G,m)-P_{DP}(G,m)$ for a fixed graph $G$, they develop techniques to evaluate $P_{DP}(G,m)$ for some classes of graphs such as chordal graphs, unicyclic graphs, theta graphs, Cartesian product graphs, joint graphs, vertex-gluings graphs, and clique-gluings graphs, etc. Zhang and Dong \cite{ZD22} give some sufficient conditions for graphs belong to $DP_{\approx}$ ($DP_{<}$, respectively) where $DP_{\approx}$ ($DP_{<}$, respectively) is the set of graphs $G$ for which there exists an $M\in \mathbb{N}$ such that $P_{DP}(G,m)=P(G,m)$ ($P_{DP}(G,m)<P(G,m)$, respectively) holds for all $m\geq M$. Their results extend Dong and Yang's results in \cite{DY21}.

\section{Bounds for DP color function of $2$-connected graphs}

In \cite{Kaul2021}, the authors obtained an upper bound for the DP color function of an arbitrary graph, by using a probabilistic argument.

\begin{lemma}[\cite{Kaul2021}]\label{up0} For any graph $G$ and all $m\in \mathbb{N}$,
$$P_{DP}(G,m)\leq\frac{m^{|V(G)|}(m-1)^{|E(G)|}}{m^{|E(G)|}}.$$
\end{lemma}

 For a connected graph, Kaul and Mudrock \cite{Kaul2021} gave the following result.

\begin{lemma}[\cite{Kaul2021}]\label{up1} For any connected graph $G$ and all $m\in \mathbb{N}$,
$$P_{DP}(G,m)=\frac{m^{|V(G)|}(m-1)^{|E(G)|}}{m^{|E(G)|}}$$ if and only if $G$ is a tree.
\end{lemma}

 Combining Lemmas \ref{up0} and \ref{up1}, one can obtain the following result easily.

\begin{theorem}\label{t} Let $G$ be a connected graph with $n$ vertices. Then for all $m\in \mathbb{N}$,
$$P_{DP}(G,m)\leq m(m-1)^{n-1}$$ where equality holds for $m\geq 2$ if and only if $G$ is a tree.
\end{theorem}

By considering the effect of the edge or vertex deletion on the DP color function, we will give two new proofs for
Theorem \ref{t}.

\begin{theorem}\label{3}
Let $G$ be a graph with $n$ vertices and $u,v$ be two distinct vertices in  $V(G)$ with $uv\not\in E(G)$.  If $G'=G+\{uv\}$, then for all $m\in \mathbb{N}$,
$$P_{DP}(G',m)\leq P_{DP}(G,m)\frac{m-1}{m}.$$
\end{theorem}

\begin{proof}Suppose that $\mathcal{H}=(L,H)$ is an arbitrary full $m$-fold cover of $G$.
Let $L'=L$ and $H'=H+E(L(u),L(v))$ where $E(L(u),L(v))$ is a perfect matching between $L(u)$ and $L(v)$ chosen uniformly at random from the $m!$ possible perfect matchings,  then $\mathcal{H'}=(L',H')$ is a full $m$-fold cover of $G'$. Let $t=P_{DP}(G,\mathcal{H})$ and $\mathcal{I}=\{I_1,\ldots,I_{t}\}$ be the set of all $\mathcal{H}$-colorings of $G$.

For each $i\in[t]$, let $E_i$ be the event that $I_i$ is also an $\mathcal{H'}$-coloring of $G'$. When $I_i\cap L(u)$ is not adjacent to $I_i\cap L(v)$ in $H'$, the event $E_i$ occurs, so
$$Pr[E_i]=1-\frac{1}{m}.$$
Let $X_i$ be the random variable that is one if $E_i$ occurs and zero otherwise. Let $X=\sum_{i=1}^{t}X_i$, then $X$ is the random variable which equals $P_{DP}(G',\mathcal{H'})$. By the linearity of expectation,  the expectation of $X$ is
$$E[X]=\sum\limits_{i=1}^{t}E[X_i]=P_{DP}(G,\mathcal{H})(1-\frac{1}{m}).$$
Then, combining the arbitrariness of $\mathcal{H}=(L,H)$, we have $$P_{DP}(G',m)\leq P_{DP}(G,m)\frac{m-1}{m}.$$
The proof is complete.
\end{proof}

\begin{corollary}\label{4}
Let $G$ be a graph with $n$ vertices and $u,v$ be two distinct vertices in  $V(G)$ with $uv\not\in E(G)$.  If $G'=G+\{uv\}$, then for all $m\in \mathbb{N}$ and $m\geq \max\{2,\chi_{DP}(G)\}$,
$$P_{DP}(G',m)< P_{DP}(G,m).$$
\end{corollary}

\begin{proof} When $m\geq \max\{2,\chi_{DP}(G)\}$, we have $0<1-1/m<1$ and $P_{DP}(G,m)>0$,  the corollary is straightforward from Theorem \ref{3}.
\end{proof}

By using Corollary \ref{4}, we give a new proof of Theorem \ref{t} as follow.

\begin{proof}({\bf proof of Theorem \ref{t}}) Let $T$ be a spanning tree of $G$, then $P_{DP}(T,m)=m(m-1)^n$ for all $m\in \mathbb{N}$. From Proposition 2.3 in \cite{Abe}, $\chi_{DP}(G)\leq 2$ if and only if $G$ is a tree. We discuss the two cases as follow.

{\bf Case 1} $|E(T)|=|E(G)|$.
In this case, $G\cong T$, $P_{DP}(G,m)=m(m-1)^{n-1}$  for all $m\in \mathbb{N}$.

{\bf Case 2} $|E(T)|<|E(G)|$.
In this case $G$ is not a tree, so $\max\{2,\chi_{DP}(G)\}=\chi_{DP}(G)\geq 3$.
\\When $m\geq \chi_{DP}(G)$, then $P_{DP}(G,m)< P_{DP}(T,m)$ from Corollary \ref{4}. When $2\leq m< \chi_{DP}(G)$, then $P_{DP}(G,m)=0< m(m-1)^{n-1}$. When $m=1$, then $P_{DP}(G,m)=0\leq m(m-1)^{n-1}$.

Summarizing the above, the theorem follows.
\end{proof}

\begin{theorem}\label{6}
Let $G$ be a graph with $n$ vertices, $w\in V(G)$ and $d_G(w)=d$, then for all $m\in \mathbb{N}$,$$P_{DP}(G,m)\leq m(1-\frac{1}{m})^dP_{DP}(G-\{w\},m).$$
\end{theorem}

\begin{proof}Suppose that $\mathcal{H'}=(L',H')$ is an arbitrary full $m$-fold cover of $G-\{w\}$, and $N_G(w)=\{v_1,\dots,v_d\}$. Let $L(x)=L'(x)$ for all $x\in V(G-\{w\})$, $L(w)=\{(w,i): i\in [m]\}$, and $E(H)=E(H')\cup (\cup_{i=1}^{d}E_H(L(w),L(v_i)))$ where for each $i\in [d]$, $E_{H}(L(w),L(v_i))$ is a perfect matching between $L(w)$ and $L(v_i)$ chosen uniformly at random from the $m!$ possible perfect matchings for each $i\in [d]$, then $\mathcal{H}=(L,H)$ is a full $m$-fold cover of $G$. Let $\Omega$ be the family of all $\mathcal{H}=(L,H)$.
	
Let $t=P_{DP}(G-\{w\},\mathcal{H'})$ and $\mathcal{I'}=\{I_1',\ldots,I_{t}'\}$ be the set of all $\mathcal{H'}$-colorings of $G-\{w\}$. In $H$, we denote $X(I_i')$  the number of vertices in $L(w)$ that is not adjacent to any vertices in $I_i'$, then $$P_{DP}(G, \mathcal{H})=\sum_{i=1}^{t}X(I_i').$$ Notice that $d_G(w)=d$, for each vertex $u\in L(w)$, in $H$ the probability that $u$ is not adjacent to any vertices in $I_i'$ is $(1-\frac{1}{m})^d$, so $$E_{\mathcal{H}\in \Omega}[X(I_i')]=m(1-\frac{1}{m})^d.$$
Then, by the linearity of expectation, we have  $$E_{\mathcal{H}\in \Omega}[P_{DP}(G, \mathcal{H})]=\sum_{i=1}^{t}E_{\mathcal{H}\in \Omega}[X(I_i')]=m(1-\frac{1}{m})^dP_{DP}(G-\{w\},\mathcal{H'}).$$
Finally, combining the arbitrariness of $\mathcal{H'}=(L', H')$, we have $$P_{DP}(G,m)\leq m(1-\frac{1}{m})^dP_{DP}(G-\{w\},m).$$
The proof is complete.
\end{proof}

By using Theorem \ref{6}, we give a new proof of Lemma \ref{up0}, along with another new proof of Theorem \ref{t} as follows.

\begin{proof}({\bf Another proof of Theorem \ref{t}})
Let $V(G)=\{v_1,\dots,v_n\}$, $G_n=G$, $G_i=G_{i+1}-\{v_{i+1}\}$ where $i\in [n-1]$, then $G_1$ is a graph with one vertex and no edges. By Theorem \ref{6}, for each $i\in [n-1]$, we have
$$P_{DP}(G_i,m)\leq m(1-\frac{1}{m})^{d_i}P_{DP}(G_{i-1},m),$$ in which $d_i=d_{G_i}(v_i).$ Then,
$$P_{DP}(G,m)\leq m^{n-1}(1-\frac{1}{m})^{\sum_{i=2}^nd_i}P_{DP}(G_1,m)=m^n(1-\frac{1}{m})^{\sum_{i=1}^nd_i}.$$
For $\sum_{i=1}^nd_i=|E(G)|$,  we have
$$P_{DP}(G,m)\leq \frac{m^{\left|V(G)\right|}(m-1)^{\left|E(G)\right|}}{m^{\left|E(G)\right|}}.$$

If $G$ is a connected graph with $n$ vertices, then $\mid E(G)\mid\geq n-1$, with equality holds if and only if $G$ is a tree.
Hence
$$P_{DP}(G,m)\leq m(m-1)^{n-1},$$ combining with that $\chi_{DP}(G)\leq 2$ if and only if $G$ is a tree,  the proof is complete.
\end{proof}	

Next we focus on the upper bounds of DP color function for $2$-connected graphs.

An {\it ear} of a graph $G$ is a maximal path whose internal vertices have degree $2$ in $G$. An {\it ear decomposition} of $G$ is a decomposition $Q_0,\ldots,Q_k$ such that $Q_0$ is a cycle and $Q_i$ for $i\geq 1$ is an ear of \ $Q_0\cup\cdots \cup Q_{i-1}$. It is well known that every $2$-connected graph has an ear decomposition.

\begin{theorem}[\cite{West01}, Theorem 4.2.8]\label{ear} A graph is $2$-connected if and only if it has an ear decomposition. Furthermore, every cycle in a $2$-connected graph is the initial cycle in some ear decomposition.
\end{theorem}

In order to get the upper bound for the DP color function of $2$-connected graphs, we first consider the DP color function of  the graph obtained by adding an ear to a graph, then combine it with the ear decomposition of $2$-connected graphs, the result follows.

\begin{theorem}\label{2}
Let $G$ be a graph with $n$ vertices and $u,v$ be two distinct vertices in  $V(G)$.  If $G'$ is a graph obtained by adding an ear
$uw_{1}\dots w_{l}v$ of length $l+1$ $(l\geq 0)$ to $G$, then for all $m\in \mathbb{N}$,
$$P_{DP}(G',m)\leq P_{DP}(G,m)\frac{(m-1)^{l+1}}{m}.$$
\end{theorem}

\begin{proof}
When $l=0$, the result follows from Theorem \ref{3}. We assume $l\geq 1$ in the following.
Suppose that $\mathcal{H}=(L,H)$ is an arbitrary full $m$-fold cover of $G$ in which $L(x)=\{(x,i):i\in [m]\}$ for each $x\in V(G)$. Let $G^*=G'-\{w_{l}v\}$, i.e., the graph $G^*$ is obtained by adding a path $P=uw_{1}\dots w_{l}$ to $G$.
Let $L^*(x)=L(x)$ for each $x\in V(G^*)-\{w_1,\ldots, w_l\}$, $L^*(x)=\{(x,i):i\in [m]\}$ for each $x\in \{w_1,\ldots, w_l\}$, and $$H^*=H+E(L^*(u),L^*(w_1))+\sum_{i=1}^{l-1}E(L^*(w_{i}),L^*(w_{i+1}))$$ where $E(L^*(u),L^*(w_1))$ ($E(L^*(w_{i}),L^*(w_{i+1}))$, respectively) is a perfect matching between $L^*(u)$ and $L^*(w_1)$ ($L^*(w_i)$ and $L^*(w_{i+1})$, respectively) chosen uniformly at random from all possible perfect matchings,  then $\mathcal{H^*}=(L^*,H^*)$ is a full $m$-fold cover of $G^*$.

From Proposition 21 in \cite{Kaul2021} and Lemma 19 in \cite{BHK22}, one can get that $$P_{DP}(G^*,\mathcal{H^*})=P_{DP}(G,\mathcal{H})(m-1)^{l}.$$
From the arbitrariness of $\mathcal{H}=(L,H)$, we have
$$P_{DP}(G^*,m)= P_{DP}(G,m)(m-1)^{l}.$$ Because $G'=G^*+\{w_{l}v\}$ and Theorem \ref{3}, we have
$$P_{DP}(G',m)\leq P_{DP}(G^*,m)\frac{m-1}{m}= P_{DP}(G,m)\frac{(m-1)^{l+1}}{m}.$$
The proof is completed.
\end{proof}

In \cite{Kaul2021}, Kaul and Mudrock computed the DP color function of the unicyclic graph (i.e., a connected graph containing exactly one cycle) with $n$ vertices, so the DP color function of the cycle with $n$ vertices can be deduced.

\begin{lemma}[\cite{Kaul2021}, Theorem 11]\label{cycle} Let $C_n$ be the cycle with $n$ vertices.

\noindent$(i)$ If $n$ is odd, then for all $m\in \mathbb{N}$,
$$P_{DP}(C_n,m)=(m-1)^n-(m-1).$$

\noindent$(i)$ If $n$ is even, then for all $m\in \mathbb{N}$ and $m\geq 2$,
$$P_{DP}(C_n,m)=(m-1)^n-1.$$
 \end{lemma}

Now we are ready to get a tight upper bound for the DP color function of $2$-connected graphs.
\begin{theorem}\label{5}
Let $G$ be a $2$-connected graph with $n$ vertices and $G_0$ be a cycle of length $l_0$ in $G$.

\noindent$(i)$ If $G_0$ is an odd cycle, then for all $m\in\mathbb{N}$ and $m\geq 3$
$$P_{DP}(G,m)\leq (m-1)^n-(m-1)^{n-l_{0}+1},$$
where equality holds if and only if $G\cong G_0$.

\noindent$(ii)$ If $G_0$ is an even cycle, then for all $m\in\mathbb{N}$ and $m\geq 3$,
$$P_{DP}(G,m)\leq (m-1)^n-(m-1)^{n-l_{0}},$$
where equality holds if and only if $G\cong G_0$.
\end{theorem}
\begin{proof}
Since $G$ is a $2$-connected graph, $G$ contains a cycle, and the DP-chromatic number of a cycle is three.
From Theorem \ref{ear}, $G$ has an ear decomposition $Q_0,\ldots,Q_k$ such that $Q_0\cong G_0$ is the cycle of length $l_0$ and $Q_i$ is an ear of $Q_0\cup\cdots \cup Q_{i-1}$ for $i\geq 1$. Suppose that ear $Q_i$ has length $l_i+1$ $(l_i\geq 0)$ for $1\leq i\leq k$, then we have $\sum_{i=0}^{k}l_{i}=n$.

By Theorems \ref{2} and Lemma \ref{cycle},
if $G_0$ is an odd cycle,
\begin{eqnarray*}
P_{DP}(G,m)&\leq& P_{DP}(G_{0},m)\prod_{i=1}^{k}\frac{(m-1)^{l_{i}+1}}{m}\\
&=&\big((m-1)^{l_{0}}-(m-1)\big)\frac{(m-1)^{n-l_{0}+k}}{m^k}\\
&=&\frac{(m-1)^{n+k}-(m-1)^{n-l_{0}+k+1}}{m^k}\\
&\leq& \frac{(m-1)^{n+k}-(m-1)^{n-l_{0}+k+1}}{(m-1)^k}\\
&=&(m-1)^n-(m-1)^{n-l_{0}+1},
\end{eqnarray*}
where the next to the last equalities hold if and only if $k=0$, i.e., $G\cong G_0$ is an $n$-vertex odd cycle.
With a similar argument, if $G_0$ is an even cycle,
\begin{eqnarray*}
P_{DP}(G,m)&\leq& P_{DP}(G_{0},m)\prod_{i=1}^{k}\frac{(m-1)^{l_{i}+1}}{m}\\
&=&\big((m-1)^{l_{0}}-1\big)\frac{(m-1)^{n-l_{0}+k}}{m^k}\\
&=&\frac{(m-1)^{n+k}-(m-1)^{n-l_{0}+k}}{m^k}\\
&\leq& \frac{(m-1)^{n+k}-(m-1)^{n-l_{0}+k}}{(m-1)^k}\\
&=&(m-1)^n-(m-1)^{n-l_{0}},
\end{eqnarray*}
where equality holds if and only if $G\cong G_0$ is an $n$-vertex even cycle.
The proof is completed.
\end{proof}

\begin{theorem}Let $G$ be a $2$-connected graph with $n$ vertices.

\noindent$(i)$ If $n$ is odd, then for all $m\in\mathbb{N}$ and $m\geq 3$,
$$P_{DP}(G,m)\leq (m-1)^n-(m-1),$$
where equality holds if and only if $G$ is an odd cycle with $n$ vertices.

\noindent$(ii)$ If $n$ is even, then for all $m\in\mathbb{N}$ and $m\geq 3$,
$$P_{DP}(G,m)\leq (m-1)^n-1,$$
where equality holds if and only if $G$ is an even cycle with $n$ vertices.
\end{theorem}

\begin{proof} Because every $2$-connected graph contains a cycle, the theorem follows from Theorem \ref{5}.
\end{proof}

\section{Canonical labelings of $\mathcal{H}$}

We begin this section by giving examples which gives negative answer to Question \ref{q}. Next we
introduce some conclusions in \cite{Kaul2021} that will be used in our later proof, then we give positive answer to Question \ref{p} for two types of graphs.

Let $G$ and $H$ be two vertex disjoint graphs, the \emph{join} $G\vee H$ of $G$ and $H$ is obtained from $G\cup H$ by joining every vertex of $G$
to every vertex of $H$. The join $C_{n}\vee K_1$ of a cycle with $n$ vertices $C_n$ and a single vertex is called a \emph{wheel} with $n$ spokes and denoted $W_n$. A \emph{theta} graph $\theta(r,s,t)$ $(r\geq 1, s,t\geq 2)$ is a graph obtained by joining two vertices by three internally disjoint
paths of lengths $r,s$ and $t$. For the wheel graph, unicyclic graph, cycle graph and theta graph, their chromatic polynomials can be found in \cite{DKT05}.

\begin{lemma}[\cite{DKT05}]\label{16}$(i)$ For the wheel $W_n$ $(n\geq 3)$, $$P(W_n,m)=m((m-2)^n+(-1)^n(m-2)).$$
$(ii)$ For a unicyclic graph $G$ with $n$ vertices containing a cycle $C_i$  $(i\geq 3)$, $$P(G,m)=(m-1)^{n}+(-1)^i(m-1)^{n-i+1}.$$
$(iii)$ For the $n$-cycle $C_n$ $(n\geq 3)$, $$P(C_n,m)=(m-1)^{n}+(-1)^n(m-1).$$
$(iv)$ For the theta graph $\theta(r,s,t)$ $(r\geq 1, s,t\geq 2)$,
\begin{eqnarray*}
P(G,m)&=&\frac{(m-1)^{r+s+t}+(-1)^{s+t}(m-1)^{r+1}+(-1)^{r+t}(m-1)^{s+1}}{m}\\
&&+\frac{(-1)^{r+s}(m-1)^{t+1}+(-1)^{r+s+t}(m-1)^2+(-1)^{r+s+t+1}(m-1)}{m}.
\end{eqnarray*}	
\end{lemma}

In a cover $\mathcal{H}=(L,H)$ of a graph $G$, the \emph{cross-edges} are the edges of $H$ connecting distinct parts of the partition $\{L(v):v\in V(G)\}$, we denote $E_c$ the set of all cross-edges in $H$, and denote $H[E_c]$ the edge-induced subgraph of $H$ induced by $E_c$.

\begin{example}Let $\mathcal{H}_1=(L_1,H_1)$ be a $3$-fold cover of $W_4$, $V(W_4)=\{x,y,z,u,v\}$ and $L_1(w)=\{(w,i):i\in[3]\}$ for each $w\in V(W_4)$. If $H_1[E_c]$ is the graph as shown in Figure \ref{figure 1}, then $P_{DP}(W_4,\mathcal{H}_1)=P(W_4,3)=6$. We list all $\mathcal{H}_1$-colorings as follows,
\begin{eqnarray*}
\{(x,1),(u,2),(y,3),(v,2),(z,2)\},&\{(x,2),(u,1),(y,3),(v,1),(z,1)\},\\
\{(x,3),(u,1),(y,2),(v,1),(z,1)\},&\{(x,1),(u,3),(y,2),(v,3),(z,3)\},\\
\{(x,2),(u,3),(y,1),(v,3),(z,3)\}, &\{(x,3),(u,2),(y,1),(v,2),(z,2)\}.
\end{eqnarray*}
But clearly $\mathcal{H}_1$ has no canonical labelings.
\end{example}

\begin{example}Let $\mathcal{H}_2=(L_2,H_2)$ be a $4$-fold cover of $W_4$, $V(W_4)=\{x,y,z,u,v\}$ and $L_2(w)=\{(w,i):i\in[4]\}$ for each $w\in V(W_4)$. If $H_2[E_c]$ is the graph as shown in Figure \ref{figure 2}, then $P_{DP}(W_4,\mathcal{H}_2)=P(W_4,4)=72$.  We list $18$ of them that are all $\mathcal{H}_2$-colorings containing $(x,1)$ as follows,
\begin{eqnarray*}
\{(x,1),(u,2),(y,3),(v,2),(z,1)\},&\{(x,1),(u,2),(y,3),(v,2),(z,4)\},\\
\{(x,1),(u,2),(y,3),(v,4),(z,1)\},&\{(x,1),(u,2),(y,3),(v,4),(z,2)\},\\
\{(x,1),(u,2),(y,4),(v,2),(z,1)\},&\{(x,1),(u,2),(y,4),(v,3),(z,1)\},\\
\{(x,1),(u,2),(y,4),(v,3),(z,2)\},&\{(x,1),(u,3),(y,2),(v,3),(z,1)\},\\
\{(x,1),(u,3),(y,2),(v,4),(z,1)\},&\{(x,1),(u,3),(y,2),(v,4),(z,3)\},\\
\{(x,1),(u,3),(y,4),(v,2),(z,1)\},&\{(x,1),(u,3),(y,4),(v,2),(z,3)\},\\
\{(x,1),(u,3),(y,4),(v,3),(z,1)\},&\{(x,1),(u,3),(y,4),(v,3),(z,2)\},\\
\{(x,1),(u,4),(y,2),(v,3),(z,4)\},&\{(x,1),(u,4),(y,2),(v,4),(z,3)\},\\
\{(x,1),(u,4),(y,3),(v,2),(z,4)\},&\{(x,1),(u,4),(y,3),(v,4),(z,2)\}.
\end{eqnarray*}
But $\mathcal{H}_2$ has no canonical labelings. \end{example}

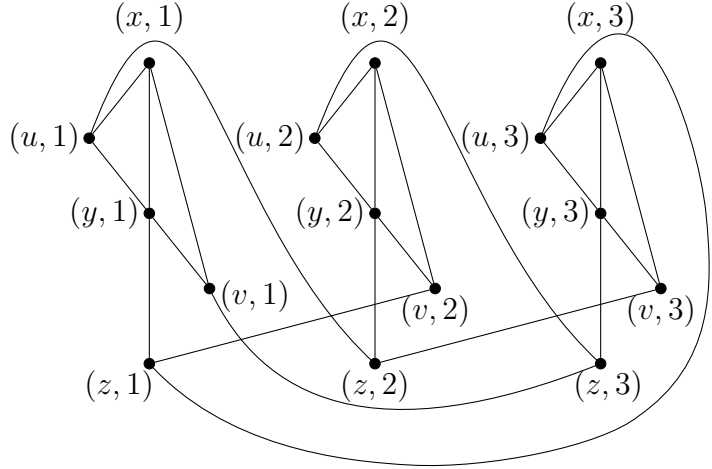
\begin{figure}[ht]
\begin{center}
\begin{tikzpicture}
[p/.style={circle,draw=black,fill=black,inner sep=1.4pt}]
\node (x1) at (0,4)[p]{}; \node (x2) at (3,4)[p]{}; \node (x3) at (6,4)[p]{};
\node (u1) at (-0.8,3)[p]{}; \node (u2) at (2.2,3)[p]{}; \node (u3) at (5.2,3)[p]{};
\node (y1) at (0,2)[p]{}; \node (y2) at (3,2)[p]{}; \node (y3) at (6,2)[p]{};
\node (v1) at (0.8,1)[p]{}; \node (v2) at (3.8,1)[p]{}; \node (v3) at (6.8,1)[p]{};
\node (z1) at (0,0)[p]{}; \node (z2) at (3,0)[p]{}; \node (z3) at (6,0)[p]{};
\draw (0,4.6) node {$(x,1)$}; \draw (3,4.6) node {$(x,2)$}; \draw (6,4.6) node {$(x,3)$};
\draw (-1.4,3) node {$(u,1)$}; \draw (1.6,3) node {$(u,2)$}; \draw (4.6,3) node {$(u,3)$};
\draw (-0.6,2) node {$(y,1)$}; \draw (2.4,2) node {$(y,2)$}; \draw (5.4,2) node {$(y,3)$};
\draw (1.4,0.9) node {$(v,1)$}; \draw (3.8,0.7) node {$(v,2)$}; \draw (6.8,0.7) node {$(v,3)$};
\draw (-0.4,-0.3) node {$(z,1)$}; \draw (3,-0.3) node {$(z,2)$}; \draw (6.1,-0.3) node {$(z,3)$};
\draw (x1)--(u1); \draw (x1)--(y1); \draw (x1)--(v1); \draw (u1)--(y1); \draw (y1)--(v1); \draw (y1)--(z1);
\draw (x2)--(u2); \draw (x2)--(y2); \draw (x2)--(v2); \draw (u2)--(y2); \draw (y2)--(v2); \draw (y2)--(z2);
\draw (x3)--(u3); \draw (x3)--(y3); \draw (x3)--(v3); \draw (u3)--(y3); \draw (y3)--(v3); \draw (y3)--(z3);
\draw (-0.8,3) .. controls (0.5,6.5) and (0.8,2) .. (3,0); \draw (2.2,3) .. controls (3.5,6.5) and (3.8,2) .. (6,0); \draw (5.2,3) .. controls (6.4,6.2) and (7.3,3) .. (7.4,2) .. controls (7.5,1) and (7.5,0) .. (6.5,-0.7) .. controls (6,-1.1) and (4.5,-1.4) .. (3.5,-1.35) .. controls (2.5,-1.3) and (1,-1.1) .. (0,0);
\draw (0.8,1).. controls (1.5,-0.5) and (3,-1.2) .. (6,0);
\draw (v2)--(z1); \draw (v3)--(z2);
\end{tikzpicture}
\caption{The subgraph $H_1[E_c]$.}
\label{figure 1}
\end{center}
\end{figure}

\begin{figure}[h!]
\begin{center}
\begin{tikzpicture}
[p/.style={circle,draw=black,fill=black,inner sep=1.4pt}]
\node (x1) at (0,4)[p]{}; \node (x2) at (3,4)[p]{}; \node (x3) at (6,4)[p]{}; \node (x4) at(9,4)[p]{};
\node (u1) at (-0.8,3)[p]{}; \node (u2) at (2.2,3)[p]{}; \node (u3) at (5.2,3)[p]{}; \node (u4) at(8.2,3)[p]{};
\node (y1) at (0,2)[p]{}; \node (y2) at (3,2)[p]{}; \node (y3) at (6,2)[p]{}; \node (y4) at(9,2)[p]{};
\node (v1) at (0.8,1)[p]{}; \node (v2) at (3.8,1)[p]{}; \node (v3) at (6.8,1)[p]{}; \node (v4) at(9.8,1)[p]{};
\node (z1) at (0,0)[p]{}; \node (z2) at (3,0)[p]{}; \node (z3) at (6,0)[p]{}; \node (z4) at(9,0)[p]{};
\draw (0,4.5) node {$(x,1)$}; \draw (3,4.5) node {$(x,2)$}; \draw (6,4.5) node {$(x,3)$}; \draw (9,4.5) node {$(x,4)$};
\draw (-1.4,3) node {$(u,1)$}; \draw (1.6,3) node {$(u,2)$}; \draw (4.6,3) node {$(u,3)$}; \draw (7.6,3) node {$(u,4)$};
\draw (-0.6,2) node {$(y,1)$}; \draw (2.4,2) node {$(y,2)$}; \draw (5.4,2) node {$(y,3)$}; \draw (8.4,2) node {$(y,4)$};
\draw (1.1,0.7) node {$(v,1)$}; \draw (4.1,0.7) node {$(v,2)$}; \draw (7.1,0.7) node {$(v,3)$}; \draw (10.1,0.7) node {$(v,4)$};
\draw (0,-0.3) node {$(z,1)$}; \draw (3,-0.3) node {$(z,2)$}; \draw (6,-0.3) node {$(z,3)$}; \draw (9,-0.3) node {$(z,4)$};
\draw (x1)--(u1); \draw (x1)--(y1); \draw (x1)--(v1); \draw (u1)--(y1); \draw (y1)--(v1); \draw (y1)--(z1); \draw (v1)--(z1);
\draw (x2)--(u2); \draw (x2)--(y2); \draw (x2)--(v2); \draw (u2)--(y2); \draw (y2)--(v2); \draw (y2)--(z2); \draw (v2)--(z2);
\draw (x3)--(u3); \draw (x3)--(y3); \draw (x3)--(v3); \draw (u3)--(y3); \draw (y3)--(v3); \draw (y3)--(z3); \draw (v3)--(z3);
\draw (x4)--(u4); \draw (x4)--(y4); \draw (x4)--(v4); \draw (u4)--(y4); \draw (y4)--(v4); \draw (y4)--(z4); \draw (v4)--(z4);
\draw (-0.8,3) .. controls (0.5,6.5) and (0.8,2) .. (3,0); \draw (2.2,3) .. controls (3.5,6.5) and (3.8,2) .. (6,0); \draw (5.2,3) .. controls (6.5,6.5) and (6.8,2) .. (9,0); \draw (8.2,3) .. controls (9.4,6.2) and (10.5,3) .. (10.6,2) .. controls (10.7,1) and (10.7,0) .. (9.5,-0.5) .. controls (7.5,-1) and (4.5,-0.9) .. (0,0);
\end{tikzpicture}
\caption{The subgraph $H_2[E_c]$.}
\label{figure 2}
\end{center}
\end{figure}
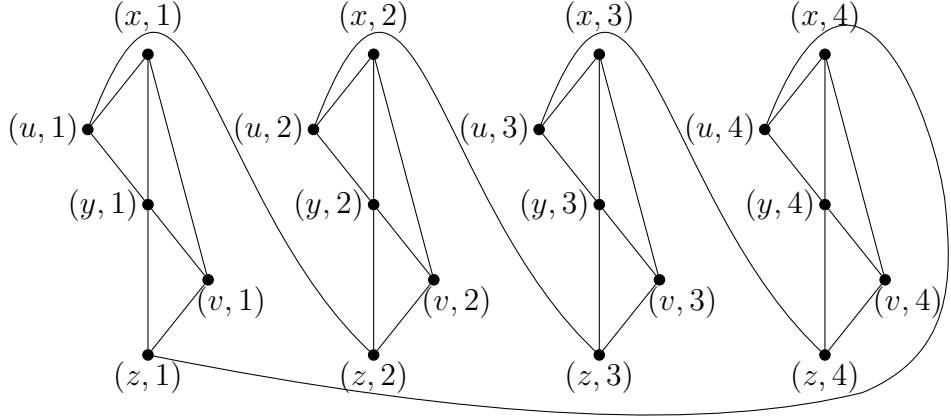

We note that in the above two examples $m=3$ or $4$, and we can't extend $m$ to larger one for the graph $W_4$.
So we consider whether Question \ref{p} has a positive answer for each graph. In fact, there are some types of graphs, for which Question \ref{p} has a positive answer.

\begin{proposition}[\cite{Kaul2021}]\label{8}
If $T$ is a tree and $\mathcal{H}=(L,H)$ is a full $m$-fold cover of $T$ where $m\geq 1$, then $\mathcal{H}$ has a canonical labeling.
\end{proposition}

In the following, we find two more examples to affirm Question \ref{p}.

\begin{lemma}[\cite{Kaul2021}]\label{9}
Let $G$ be a graph with $e=uv\in E(G)$. For each $(i,j)\in [m]\times [m]$, let $C_m^{(i,j)}$ be the set of proper $m$-coloring of $G-\{e\}$ that color $u$ with $i$ and $v$ with $j$. Then,

\noindent$(i)$ there is an $r\in \mathbb{N}$ such that  $\lvert C_m^{(i,i)}\rvert=r$ for each $i\in[m]$.

\noindent$(ii)$	there is a $t\in \mathbb{N}$ such that  $\lvert C_m^{(i,j)}\rvert=t$ whenever $i\neq j$ and $i,j\in[m]$.\\Consequently, $mr=P(G-\{e\},m)-P(G,m)$ and $m(m-1)t=P(G,m)$.
\end{lemma}

From Lemma \ref{9} and the definition of canonical labeling, we obtain the following lemma.

\begin{lemma}\label{10}
Let $G$ be a graph and $\mathcal{H}=(L,H)$ be a full $m$-fold cover of $G$ with $m\geq 2$. Suppose $e\in E(G)$ and $e=uv$. Let $H'=H-E_H(L(u),L(v))$ so that $\mathcal{H'}=(L,H')$ is a full $m$-fold cover of $G-\{e\}$. For each $(i,j)\in [m]\times [m]$, let $\mathcal{H'}_{(i,j)}$ be the set of $\mathcal{H'}$-coloring that contain $(u,i)$ and $(v,j)$. If $\mathcal{H'}$ has a canonical labeling, then we have
\\$(i)$ when $i=j$, $$\lvert \mathcal{H'}_{(i,j)}\rvert=\frac{P(G-\{e\},m)-P(G,m)}{m};$$
\\$(ii)$ when $i\neq j$, $$\lvert \mathcal{H'}_{(i,j)}\rvert=\frac{P(G,m)}{m(m-1)}.$$
Furthermore, suppose $P=\{(i,j): (u,i)(v,j)\in E_H(L(u),L(v))\}$, then
$$P_{DP}(G,\mathcal{H})=P_{DP}(G',\mathcal{H'})-\sum_{(i,j)\in P}\lvert \mathcal{H'}_{(i,j)}\rvert.$$
\end{lemma}

\begin{lemma}[\cite{Kaul2021}]\label{11}
Let $G$ be a graph and $\mathcal{H}=(L,H)$ be a full $m$-fold cover of $G$ with $m\geq 3$. Suppose $\alpha_1$$\alpha_2$$\alpha_3$ is a path of length two in $G$ and $\alpha_1\alpha_3\notin E(G)$. Let $e_1=\alpha_1\alpha_2$, $e_2=\alpha_2\alpha_3$. Then, let $G_0=G-\{e_1, e_2\}$, $G_1=G-\{e_1\}$, $G_2=G-\{e_2\}$, and $G^*$ be the graph obtained from $G$ by adding an edge between $\alpha_1$ and $\alpha_3$. Let $H'=H-(E_H(L(\alpha_1), L(\alpha_2))\cup E_H(L(\alpha_2), L(\alpha_3)))$ so that $\mathcal{H'}=(L,H')$ is an $m$-fold cover of $G_0$. Suppose that $\mathcal{H'}$ has a canonical labeling. Let
\begin{eqnarray*}
&A_1&=P(G_0,m)-P(G,m),\\
&A_2&=P(G_0,m)-P(G_2,m)+\frac{1}{m-1}P(G,m),\\
&A_3&=P(G_0,m)-P(G_1,m)+\frac{1}{m-1}P(G,m),\\
&A_4&=\frac{1}{m-1}(P(G_1,m)+P(G_2,m)+P(G^*,m)-P(G,m)), and\\
&A_5&=\frac{1}{m-1}(P(G_1,m)+P(G_2,m)-\frac{1}{m-2}P(G^*,m)).\\
\end{eqnarray*} Then,
$$P_{DP}(G,\mathcal{H})\geq P(G_0,m)-\max\{A_1, A_2, A_3, A_4, A_5\}.$$ Moreover, there exists an $m$-fold cover of $G$, $\mathcal{H^*}$, such that $$P_{DP}(G,\mathcal{H^*})=P(G_0,m)-\max\{A_1, A_2, A_3, A_4, A_5\}.$$
\end{lemma}

From Lemma \ref{11} and its proof in \cite{Kaul2021}, we get Lemma \ref{12}.

\begin{lemma}\label{12}
Under the condition of Lemma \ref{11}. Let $H''$ be the graph with $V(H'')=\bigcup_{i=1}^{3}L(\alpha_i)$ and $E(H'')=E_H(L(\alpha_1),L(\alpha_2))\cup E_H(L(\alpha_2),L(\alpha_3))$. Clearly $H''$ can be decomposed into $m$ vertex disjoint paths on three vertices. Take any one of $m$ paths, let it be $(\alpha_1,i)(\alpha_2,j)(\alpha_3,k)$ where $i,j,k\in [m]$, then we have five cases for $i,j,k$, that are (1) $i=j=k$, (2) $i=j$ and $j\neq k$, (3) $i\neq j$ and $j=k$, (4) $i\neq j$ and $i=k$, (5) $i$, $j$ ,$k$ are pairwise distinct.
Let $\mathcal{H'}_{(i,j,k)}$ be the set of $\mathcal{H'}$-coloring that contains at least one edge of the path $(\alpha_1,i)(\alpha_2,j)(\alpha_3,k)$. Then $\lvert \mathcal{H'}_{(i,j,k)}\rvert=A_q/m$ when $i,j,k$ satisfy case $q$, $(1\leq q\leq 5)$.
Furthermore, we suppose that for $q\in [5]$, there are $m_q$ paths of case $q$ in the $m$ paths. Then $\sum_{q=1}^{5}m_q=m$ and
$$P_{DP}(G,\mathcal{H})=P(G_0,m)-\frac{1}{m}\sum\limits_{q=1}^{5}m_qA_q.$$
\end{lemma}

\begin{theorem}\label{13}
Let $G$ be a unicyclic graph with $n$ vertices containing a cycle $C$ on $g$ vertices where $g\geq 3$ and $\mathcal{H}=(L,H)$ be a full $m$-fold cover of $G$. For each $m\geq 2$, if $P_{DP}(G, \mathcal{H})=P(G,m)$, then $\mathcal{H}$ has a canonical labeling.
\end{theorem}

\begin{proof}
Suppose $e\in E(C)$ and $e=uv$. Let $G'=G-\{e\}$ and $\mathcal{H'}=(L,H')$ where $H'=H-E_H(L(u),L(v))$. Then, $G'$ is a tree and $\mathcal{H'}$ is a full $m$-fold cover of $G'$.  Proposition \ref{8} implies that $\mathcal{H'}$ has a canonical labeling. Let $P=\{(i,j): (u,i)(v,j)\in E_H(L(u),L(v))\}$, $P_1=\{(i,j)\in P\ \mbox{and}\ i=j\}$, and $P_2=\{(i,j)\in P \ \mbox{and}\ i\neq j\}$. Suppose that $|P_2|=t$, then $|P_1|=m-t$. By Lemma \ref{10}, we have that for $m\geq 2$
\begin{eqnarray*}
P_{DP}(G,\mathcal{H})&=&P_{DP}(G',\mathcal{H'})-\sum_{(i,j)\in P}\lvert \mathcal{H'}_{(i,j)}\rvert\\
&=&P(G',m)-t\frac{P(G,m)}{m(m-1)}-(m-t)\frac{P(G',m)-P(G,m)}{m}.
\end{eqnarray*}
For $P(G,m)=(m-1)^{n}+(-1)^g(m-1)^{n-g+1}$ and $P(G',m)=m(m-1)^{n-1}$, we have
$$P_{DP}(G,\mathcal{H})=(m-1)^n+(-1)^g(m-t-1)(m-1)^{n-g}.$$
If $P_{DP}(G, \mathcal{H})=P(G,m)$, then $t=0$ which implies $\mathcal{H}$ is a canonical labeling.
\end{proof}

\begin{theorem}\label{15}
Let $G=\theta(r,s,t)$ $(r\geq 1, s,t\geq 2)$ and $\mathcal{H}=(L,H)$ be a full $m$-fold cover of $G$. For each $m\geq 3$, if $P_{DP}(G,\mathcal{H})=P(G,m)$, then $\mathcal{H}$ has a canonical labeling.
\end{theorem}	

\begin{proof}
Let $\alpha_2$ be one of the common ends of the three paths of $G$, let $\alpha_1$ and $\alpha_3$ be the vertices in the path of length $s$ and $t$ respectively, that are adjacent to $\alpha_2$. Clearly $\alpha_1\alpha_2\alpha_3$ is a path of length two in $G$ and $\alpha_1\alpha_3\not\in E(G)$.
We define $e_1, e_2, G_0, G_1, G_2, G^{*}, \mathcal{H'}, H''$ and $m_q (1\leq q\leq 5)$ as they are defined in the statement of Lemmas \ref{11} and \ref{12}.

Then $G_0$ is a tree and $\mathcal{H'}=(L,H')$ is a full $m$-fold cover of $G_0$. By Proposition \ref{8}, $\mathcal{H'}$ has a canonical labeling.  So we can name the vertices of $H$ so that $L(x)=\{(x,j): j\in [m]\}$ for each $x\in V(H)$ and $(x,j)(y,j)\in E(H)$ for each $j\in [m]$ whenever $xy\in E(G_0)$.

By computing, we get that
\begin{equation}\label{e1}P(G_0,m)=m(m-1)^{r+s+t-2},\end{equation}
\begin{equation}\label{e2}P(G_1,m)=(m-1)^{r+s+t-1}+(-1)^{r+t}(m-1)^s,\end{equation}
\begin{equation}\label{e3}P(G_2,m)=(m-1)^{r+s+t-1}+(-1)^{r+s}(m-1)^t,\end{equation}
\begin{eqnarray}\label{e4}
P(G^*,m)=\left\{
\begin{aligned}
&P(G,m)-P(\theta(r+1,s-1,t-1),m), ~~~\mbox{when}~~ s\geq 3 ~\mbox{or}~t\geq 3;\\
&P(G,m)-P(C_{r+2},m), ~~~~~~~~~~~~~~~~~~~~~~~~~~~~~~\mbox{when}~~ s=t=2.
\end{aligned}
\right.
\end{eqnarray}

By Lemmas \ref{12} and \ref{11}, we have that
\begin{eqnarray}\label{e6}
P_{DP}(G,\mathcal{H})&=&P(G_0,m)-\frac{1}{m}\sum\limits_{q=1}^{5}m_qA_q\nonumber\\
&=&\frac{m-m_1-m_2-m_3}{m}P(G_0,m)+\frac{(m-1)m_1-m_2-m_3+m_4}{m(m-1)}P(G,m)\nonumber\\
&&+\frac{(m-1)m_3-m_4-m_5}{m(m-1)}P(G_1,m)+\frac{(m-1)m_2-m_4-m_5}{m(m-1)}P(G_2,m)\nonumber\\
&&+\frac{m_5-(m-2)m_4}{m(m-1)(m-2)}P(G^*,m).
\end{eqnarray}

For simplicity, we let $u=m-1$. Then combining equations \eqref{e1}-\eqref{e6} with Lemma \ref{16}$(iv)$, we have that
when $s\geq 3$ or $t\geq 3$,
\begin{eqnarray}\label{e7}
P_{DP}(G,\mathcal{H})&=&\frac{u^{r+s+t+1}+u^{r+s+t}+(-1)^{r+t}(u^{s+2}+u^{s+1})+(-1)^{s+t}(u^{r+2}+u^{r+1})}{(u+1)^2}\nonumber\\
&&+\frac{(-1)^{r+s}(u^{t+2}+u^{t+1})+(-1)^{r+t+1}(m_2+m_4+m_5)(u^{s+1}+u^s)}{(u+1)^2}\nonumber\\
&&+\frac{(-1)^{s+t+1}(m_2+m_3+m_5)(u^{r+1}+u^r)+(-1)^{r+s+1}(m_3+m_4+m_5)(u^{t+1}+u^t)}{(u+1)^2}\nonumber\\
&&+\frac{(-1)^{r+s+t}u^3+(-1)^{r+s+t+1}(m_2+m_3+m_4+m_5)u^2}{(u+1)^2}\nonumber\\
&&+\frac{(-1)^{r+s+t}(m_5-1)u+(-1)^{r+s+t}(m_2+m_3+m_4+2m_5)}{(u+1)^2};
\end{eqnarray}
when $s=t=2$,
\begin{eqnarray}\label{e77}
P_{DP}(G,\mathcal{H})&=&\frac{u^{r+6}-u^{r+4}+u^{r+3}-u^{r+1}-(m_2+m_3+m_5)(u^{r+2}-u^{r})+(-1)^r2u^5}{(u+1)^2(u-1)}\nonumber\\
&&+\frac{(-1)^{r+1}(m_2+m_3+2m_4+2m_5-1)u^4+(-1)^{r+1}(m_2+m_3+m_4+m_5+3)u^3}{(u+1)^2(u-1)}\nonumber\\
&&+\frac{(-1)^{r}(2m_2+2m_3+3m_4+4m_5-1)u^2+(-1)^r(m_2+m_3+m_4+m_5+1)u}{(u+1)^2(u-1)}\nonumber\\
&&+\frac{(-1)^{r+1}(m_2+m_3+m_4+2m_5)}{(u+1)^2(u-1)}
\end{eqnarray}
and
\begin{eqnarray}\label{e8}
P(G,m)&=&P(G,u+1)\nonumber\\
&=&\frac{u^{r+s+t}+(-1)^{s+t}u^{r+1}+(-1)^{r+t}u^{s+1}+(-1)^{r+s}u^{t+1}}{u+1}\nonumber\\
&&+\frac{(-1)^{r+s+t}u^2+(-1)^{r+s+t+1}u}{u+1}.
\end{eqnarray}	
Let $f_1(u)$, $f_2(u)$ and $g(u)$ be the numerator of the equations \eqref{e7}, \eqref{e77} and \eqref{e8} respectively.
Let $h_1(u)=f_1(u)-(u+1)g(u)$ and $h_2(u)=f_2(u)-(u^2-1)g(u)$.
Because $g(u)$ is a polynomial, $P_{DP}(G,\mathcal{H})=P(G,m)$ if and only if
$h_1(u)=0,$ when $s\geq 3$ or $t\geq 3$;  $h_2(u)=0$ when $s=t=2$. In the following, we will prove that if
$h_1(u)=0$ (or $h_2(u)=0$), then $m_1=m$, $m_2=m_3=m_4=m_5=0$, i.e.,$\mathcal{H}$ has a canonical labeling.
We discuss the two cases respectively.

\noindent{\bf Case 1.} $s\geq 3$ or $t\geq 3$.

In this case
\begin{eqnarray*}
h_1(u)&=&(-1)^{r+t+1}(m_2+m_4+m_5)(u^{s+1}+u^s)+(-1)^{s+t+1}(m_2+m_3+m_5)(u^{r+1}+u^r)\\
&&+(-1)^{r+s+1}(m_3+m_4+m_5)(u^{t+1}+u^t)+(-1)^{r+s+t+1}(m_2+m_3+m_4+m_5)u^2\\
&&+(-1)^{r+s+t}m_5u+(-1)^{r+s+t}(m_2+m_3+m_4+2m_5).
\end{eqnarray*}
If $h_1(u)$ is a zero polynomial, then each coefficient of the polynomial is zero.  We note that
$\sum_{q=1}^{5}m_q=m=u+1$ and $m_q\geq 0$ for each $q\in[5]$.

Firstly, we focus on the constant term of $h_1(u)$. If the constant term of $h_1(u)$ is zero, then $m_2+m_3+m_4+2m_5\equiv 0 ~(\bmod ~u)$, i.e., $m_2+m_3+m_4+2m_5=ku$. And $k\in \{0,1,2\}$, because $m_2+m_3+m_4+2m_5\in \left[0,2u+2\right]$. According to the value of $k$, we have the following three situations.

{\bf S1:} $k=0$, i.e., $m_2+m_3+m_4+2m_5=0$.  Then we have $m_2=m_3=m_4=m_5=0$ and $m_1=m$.

{\bf S2:} $k=1$,  i.e.,$m_2+m_3+m_4+2m_5=m-1$. Then we have $m_1=m_5+1$.

{\bf S3:} $k=2$,  i.e.,$m_2+m_3+m_4+2m_5=2m-2$. Then we have $m+m_1-m_5=2$ which implies $m_1=0$, $m_5=m-2$, $m_2+m_3+m_4=2$; or $m_1=1$, $m_5=m-1$, $m_2=m_3=m_4=0$.

Clearly, in situation S1, $h_1(u)=0$ and $\mathcal{H}$ has a canonical labeling. In the following, we will prove that in situations S2 and S3, $h_1(u)$ is not a zero polynomial. We focus on the coefficient of $u$ in $h_1(u)$  and discuss the following two subcases.

\noindent{\bf Subcase 1.1.} $r\geq 2.$

In this subcase, if the coefficient of $u$ in $h_1(u)$ is zero, then $m_5+1\equiv 0 ~(\bmod ~u)$ in situation S2, and $m_5+2\equiv 0 ~(\bmod ~u)$ in situation S3.

In situation S2, if $m_5+1\equiv 0 ~(\bmod ~u)$, then $m_5=u-1=m-2$, $m_1=m-1$, and $m_1+m_5=2m-3$. Because $\sum_{q=1}^{5}m_q=m$, we have $2m-3\leq m$, then $m\leq 3$. So $m_1=2, m_5=1$, but this will not happen, we can't have only one path $(\alpha_1,i)(\alpha_2,j)(\alpha_3,k)$ where $i,j,k$ are pairwise distinct.

In situation S3, if $m_5+2\equiv 0 ~(\bmod ~u)$, then $m_5=u-2=m-3$, this contradicts $m_5=m-2$ or $m_5=m-1$.

\noindent{\bf Subcase 1.2.} $r=1.$

In this subcase, if the coefficient of $u$ in $h_1(u)$ is zero, then $m_2+m_3+2m_5+1\equiv 0 ~(\bmod ~u)$ in situation S2, and $m_2+m_3+2m_5+2\equiv 0 ~(\bmod ~u)$ in situation S3.

In situation S2, if $m_2+m_3+2m_5+1\equiv 0 ~(\bmod ~u)$, then $m_2+m_3+2m_5=u-1$, combining this with $m_2+m_3+m_4+2m_5=u$, $m_4=1$ follows. But when $m_4=1$, the coefficient of the leading term is not zero.

In situation S3, $m_5=m-2$ or $m-1$.
If $m_5=m-2$ and $m_2+m_3+2m_5+2\equiv 0 ~(\bmod ~u)$, then $m_2+m_3+2m_5=2u-2$. Combining this with $m_2+m_3+m_4=2$, we have $m_2=m_3=0$ and $m_4=2$.
 But when $m_4=2$, the coefficient of the leading term is not zero. If $m_5=m-1$, then $m_2+m_3+2m_5=2u$,  $m_2+m_3+2m_5+2\neq 0 ~(\bmod ~u)$, otherwise $u=2$. But when $u=2$, we have $m_5=2$, $m_1=1$ and $m_2=m_3=m_4=0$. This will not happen, because we can't have two paths $(\alpha_1,i)(\alpha_2,j)(\alpha_3,k)$ where $i,j,k$ are pairwise distinct when $m=3$.

Hence, in Case 1,  if $h_1(u)=0$, then $m_1=m, m_2=m_3=m_4=m_5=0$, i.e., $\mathcal{H}$ has a canonical labeling.

\noindent{\bf Case 2. $s=t=2$.}

In this case
\begin{eqnarray*}
h_2(u)&=&(-m_2-m_3-m_5)u^{r+2}+(m_2+m_3+m_5)u^r+(-1)^{r+1}(m_2+m_3+2m_4+2m_5)u^4\\
&&+(-1)^{r+1}(m_2+m_3+m_4+m_5)u^3+(-1)^r(2m_2+2m_3+3m_4+4m_5)u^2\\
&&+(-1)^r(m_2+m_3+m_4+m_5)u+(-1)^{r+1}(m_2+m_3+m_4+2m_5).
\end{eqnarray*}

The proof is similar to that for Case 1. If the constant term of $h_2(u)$ is zero, then $m_2+m_3+m_4+2m_5\equiv 0 ~(\bmod ~u)$, which is the same with that in Case 1. So we have the same three situations S1, S2, S3 with that in Case 1. Clearly, in situation S1, $h_2(u)=0$ and $\mathcal{H}$ has a canonical labeling. In the following, we will prove that in situations S2 and S3, $h_2(u)$ is not a zero polynomial. We discuss the following two subcases.

\noindent{\bf Subcase 2.1.} $r\geq 2.$

In situation S2,
\begin{eqnarray*}&&(-1)^r(m_2+m_3+m_4+m_5)u+(-1)^{r+1}(m_2+m_3+m_4+2m_5)\\
&=&(-1)^{r}u^2+(-1)^{r+1}(m_5+1)u,\end{eqnarray*}
and in situation S3,
\begin{eqnarray*}&&(-1)^r(m_2+m_3+m_4+m_5)u+(-1)^{r+1}(m_2+m_3+m_4+2m_5)\\
&=&(-1)^{r}2u^2+(-1)^{r+1}(m_5+2)u.\end{eqnarray*}
If the coefficient of $u$ in $h_2(u)$ is zero, then $m_5+1\equiv 0 ~(\bmod ~u)$ in situation S2, and $m_5+2\equiv 0 ~(\bmod ~u)$ in situation S3, which are exactly the same with that in Subcase 1.1. So with the same argument in Subcase 1.1,  we obtain that in situations S2 and S3, $h_2(u)$ is not a zero polynomial in Subcase 2.1.

\noindent{\bf Subcase 2.2.} $r=1.$

If the coefficient of $u$ in $h_2(u)$ is zero, then $$(m_2+m_3+m_5)u-(m_2+m_3+m_4+m_5)u+ku=(k-m_4)u=0.$$
In situation S2, $k=1$, so $m_4=1$; in  situation S3, $k=2$, so $m_4=2$. But no matter $m_4$ is $1$ or $2$, the leading term of $h_2(u)$ will not be zero.

Hence,  in Case 2,  if $h_2(u)=0$, then $m_1=m, m_2=m_3=m_4=m_5=0$, i.e., $\mathcal{H}$ has a canonical labeling.

Summarizing Cases 1 and 2, the theorem is obtained.
\end{proof}

By Theorems \ref{13} and \ref{15}, we know that the answer of Question \ref{p} is yes for unicyclic graphs and theta graphs when $m\geq 2$, and $m\geq 3$ respectively. Whether the answer of Question \ref{p} is yes for all graphs is still wide open.

\section*{Declarations}

{\bf Conflict of interest}  The authors have no relevant financial or non-financial interests to disclose.

\noindent{\bf Data sharing}  Data sharing not applicable to this article as no datasets were generated or analysed during
the current study.

\section*{Acknowledgements}

The authors would like to thank a referee for valuable comments and suggestions.


\begin{thebibliography}{99}
\frenchspacing

\bibitem{Abe} T. Abe, Differences between the list-coloring and DP-coloring for planar graphs, \emph{Discrete Math.}, {\bf 344} (2021), Paper No.   112471.

\bibitem{BHK22}J. Becker, J. Hewitt, H. Kaul, M. Maxfield, J. A. Mudrock, D. Spivey, S. Thomason, and T. Wagstrom, The DP color function of joins and vertex-gluings of graphs,  \emph{Discrete Math.}, {\bf345(11)} (2022), Paper No. 113093.

\bibitem{B1912} G. D. Birkhoff, A determinant formula for the number of ways of coloring a map, \emph{Ann. of Math.}, {\bf14}
(1912) 42--46.

\bibitem{DKT05} F. M. Dong, K. M. Koh, and K. L. Teo, Chromatic Polynomials and Chromaticity of Graphs, World
Scientific, Singapore, 2005.

\bibitem{DY21} F. M. Dong and Y. Yang, DP color functions versus chromatic polynomials, \emph{Adv. in Appl.
Math.}, {\bf134} (2022), Paper No. 102301.

\bibitem{Dvorak2018}  Z. Dvo\v{r}\'{a}k, L. Postle,
Correspondence coloring and its application to list-coloring planar graphs without cycles of lengths $4$ to $8$,
\emph{J. Combin. Theory, Ser. B}, {\bf 129} (2018), 38--54.



\bibitem{EEFH21} J. Engbers, A. Erey, J. Fox, and X. Y. He, Tomescu's graph coloring conjecture for $l$-connected graphs, \emph{SIAM J. Discrete. Math.}, {\bf35(2)} (2021), 1478--1502.

\bibitem{Erey08} A. Erey, Maximizing the number of $x$-colorings of $4$-chromatic graphs, \emph{Discrete Math.}, {\bf 341} (2008), 1419--1431.

\bibitem{Erey18} A. Erey, On the maximum number of colorings of a graph, \emph{J. Combin.}, {\bf 9(3)} (2018), 489--497.


\bibitem{Felix2019} L. Felix, The maximum number of colorings of graphs of given order and size: a survey, \emph{Discrete Math.}, {\bf342(10)} (2019), 2783--2791.

\bibitem{FHM19} J. Fox, X. He, and F. Manners, A proof of Tomescu's graph coloring conjecture, \emph{J. Combin.
Theory Ser. B}, {\bf 136} (2019), 204--221.


\bibitem{HKL20}C. Halberg, H. Kaul, A. Liu, J. A. Mudrock, P. Shin, and S. Thomason, On polynomial representations of the DP color function: theta graphs and their generalizations, arXiv:2012.12897, (2020).

\bibitem{KMMT21}H. Kaul, M. Maxfield, J. A. Mudrock, S. Thomason, The DP Color Function of Clique-Gluings of Graphs, \emph{Enumer. Comb. Appl.}, {\bf4(2)} (2024), Paper No. S2R11.

\bibitem{Kaul2021}  H. Kaul, J. A. Mudrock,
On the chromatic polynomial and counting DP-colorings of graphs,
\emph{Adv. in Appl. Math.}, {\bf 123} (2021), Paper No. 103121.

\bibitem{KMSS21}H. Kaul, J. A. Mudrock, G. Sharma, Q. Stratton, DP-coloring Cartesian product of graphs,  \emph{J. Graph Theory}, {\bf103(2)} (2023), 285--306.

\bibitem{KM19} F. Knox and B. Mohar, Maximum number of colourings: 5-chromatic case, \emph{Electronic J.
Combin.}, {\bf26(3)} (2019), Paper No. 3.40.

\bibitem{KM20}  F. Knox and B. Mohar, Maximum number of colourings: 4-chromatic graphs, \emph{J. Combin. Theory Ser. B}, {\bf 144} (2020), 95--118.

\bibitem{Mu22} J. A. Mudrock , A deletion-contraction relation for the DP color function, \emph{Graphs Combin.}, {\bf 38(4)} (2022), Paper No. 115.

\bibitem{MT21} J. A. Mudrock and S. Thomason, Answers to two questions on the DP color function, \emph{Electron. J. Combin.}, {\bf 28(2)} (2021), Paper No. 2.24.

\bibitem{Tomescu1990} I. Tomescu, Maximal chromatic polynomials of connected planar graphs, \emph{J. Graph Theory}, {\bf 14} (1990), 101--110.

\bibitem{Tomescu1994} I. Tomescu, Maximum chromatic polynomials of 2-connected graphs, \emph{J. Graph Theory}, {\bf 18} (1994), 329--336.

\bibitem{West01} D. B. West, Introduction to Graph Theory, second ed., Prentice Hall, New York, 2001.

\bibitem{ZD22} M. Q. Zhang, F. M. Dong, DP color functions versus chromatic polynomials (II),  \emph{J. Graph Theory},
  {\bf 103(4)} (2023),  740--761.


\end{thebibliography}
\end{document}